\documentclass[10pt,regno]{amsart}
\usepackage{amsmath,amsthm,amsfonts,amssymb,amscd,overpic,mathrsfs}
\usepackage{euscript,graphicx}
\usepackage{epstopdf,rotating}
\usepackage{color}
\newtheorem{defi}{Definition}[section]
\newtheorem{teorema}[defi]{Theorem}
\newtheorem{lema}[defi]{Lemma}
\newtheorem{prop}[defi]{Proposition}

\newtheorem{cor}[defi]{Corollary}

\newtheorem*{remark*}{Remark}

\newtheorem{claim}[defi]{Claim}

\newcommand{\texto}{\textrm}
\newcommand{\ovl}{\overline}

\newcommand{\disp}{\displaystyle}
\newcommand{\N}{\mathbb{N}}
\newcommand{\R}{\mathbb{R}}

\newcommand{\C}{\mathbb{C}}

\DeclareMathOperator{\DD}{DD}
\DeclareMathOperator{\hD}{hD}
\DeclareMathOperator{\Isom}{Isom}
\DeclareMathOperator{\diam}{diam}

\author[S. Campos]{Sara Campos}
 \address{Department of Mathematics, 
 Federal University of Juiz de Fora, 
Campus Universit\'ario - Bairro Martelos, 
Juiz de Fora  36036-900, MG, Brazil}
 \email{sara.campos@edu.ufjf.br}
\author[K. Gelfert]{Katrin~Gelfert}
 \address{Institute of Mathematics, Federal University of Rio de Janeiro, Av. Athos da Silveira Ramos 149, Cidade Universit\'aria - Ilha do Fund\~ao, Rio de Janeiro 21945-909, RJ, Brazil}
\email{gelfert@im.ufrj.br}
\thanks{We are grateful to Juan Rivera-Letelier for enlightening conversations and pointing out~\cite{RivShe:} and to Lasse Rempe-Gillen for useful comments. SC was supported by CAPES  (Brazil) and KG   by CNPq (Brazil).}
\keywords{Topological entropy, Hausdorff dimension, exceptional sets}
\title[Exceptional sets]{Exceptional sets\\for nonuniformly expanding maps}
\subjclass[2000]{
37B40, 
37C45, 
37D25, 
37F35, 
37F10 
37E05, 
}

\begin{document}
\maketitle
\begin{abstract}%
Given a rational map of the Riemann sphere and a subset $A$ of its Julia set, we study the $A$-exceptional set, that is, the set of points whose orbit does not accumulate at $A$. 
We prove that if the topological entropy of $A$ is less than the topological entropy of the full system then the $A$-exceptional set has full topological entropy. 
Furthermore, if the Hausdorff dimension of $A$ is smaller than the dynamical dimension of the system  then the Hausdorff dimension of the $A$-exceptional set is larger than or equal to the dynamical dimension, with equality in the particular case when the dynamical dimension and the Hausdorff dimension coincide.

We discuss also the case of a general conformal $C^{1+\alpha}$ dynamical system and, in particular,  certain  multimodal interval maps on their Julia set. 
\end{abstract}

\section{Introduction}
Consider a compact metric space  $(X,d)$ and a continuous transformation $f\colon X \to X$. Let $W\subset X$ be  $f$-invariant, that is, $f(W) = W$.  Given $A\subset W$, the \emph{$A$-exceptional set} in $W$ (with respect to $f|_W$) is defined to be the set
\[
E^+_{f|W}(A) := \{x\in W\colon  \overline{\mathcal{O}_f(x)}\cap A = \emptyset\},
\]
where $\mathcal{O}_f(x):= \{f^k(x)\colon  k\in \N \cup \{0\}\}$ denotes the forward orbit of $x$ by $f$. In other words, $E^+_{f|W}(A)$ is the set of points in $W$ whose forward orbit does not accumulate at $A$. In this paper we study the ``size'' of exceptional sets in terms of their topological entropy and their Hausdorff dimension. We will consider as dynamical systems rational functions of the Riemann sphere which include those with parabolic points and critical points. 

The following is our first main result stated in terms of topological entropy (we recall its definition in Section~\ref{Ent}). 
 
\begin{teorema}\label{teoentropy}
Let $f\colon \ovl{\C} \to \ovl{\C}$ be a rational function of degree $d\ge2$ on the Riemann sphere and let $J=J(f)$ be its Julia set. 

If  $A\subset J$ satisfies $h(f|_J,A) < h(f|_J)=\log d$, then  
$$
h( f|_J,E^+_{f|J}(A)) = h(f|_J)=\log d.
$$
\end{teorema}

The above result will be a consequence of a corresponding statement for entropy of a continuous shift-equivalent transformation (see Proposition~\ref{proph}).

The second result in terms of the Hausdorff dimension $\dim_{\rm H}$ uses  canonical  concepts which we briefly recall (see Section~\ref{HD} for more details).  
Given a $f$-invariant probability measure $\mu$, the {\it Hausdorff dimension of $\mu$} is defined by
$$
\dim_{\rm H}\mu := \inf \{\dim_{\rm H}Y\colon  Y\subset X \text{ and } \mu(Y) = 1\}.
$$
The {\it dynamical dimension} of $f$ is defined by   
\begin{equation}\label{def:DD}
	\DD(f|_X) := \sup_\mu\dim_{\rm H}\mu,
\end{equation}
where the supremum is taken over all ergodic measures $\mu$ with positive entropy. We will consider only maps where such measures do exist and where hence $\DD$ is well defined. Note that clearly we have $\DD(f|_X)\le \dim_{\rm H}X$. 

The following relation was established in the context of a general rational function $f$ of degree $\ge2$ of the Riemann sphere and $X=J(f)$ its Julia set (see~\cite[Chapter 12.3]{PU}) and will be fundamental for our approach. We have
\begin{equation}\label{eq:dyndimdims}
	\DD(f|_{J(f)})
	= \hD(f|_{J(f)}),
	\quad\text{ where }\quad
	\hD(f|_{J(f)})
	:=\sup_Y\dim_{\rm H}Y,
\end{equation}
where the supremum is taken over all conformal expanding repellers $Y\subset J(f)$  (we recall its definition in Section~\ref{chirepellers}), the latter number is also called the \emph{hyperbolic dimension} of $J(f)$.

\begin{teorema}\label{teoprinc}
Let $f\colon \ovl{\C} \to \ovl{\C}$ be a rational function of degree $\ge2$ on the Riemann sphere and let $J=J(f)$ be its Julia set. 

If $A\subset J$ satisfies $\dim_{\rm H} A < \DD(f|_{J})$, then  
$$
	\dim_{\rm H} E_{f|J}^+(A) \geq \DD(f|_{J}).
$$
\end{teorema}

Theorem~\ref{teoprinc} immediately implies the following.%
\footnote{Note that until recently it was unknown whether there exist a map for which $\hD(f|_J)<\dim_{\rm H}J$. Avila and Lyubich in~\cite[Theorem D]{AviLyu:07} show that for so-called Feigenbaum maps with periodic combinatorics whose Julia set has positive area one has $\hD(f|_J)<\dim_{\rm H}J=2$. They provide examples in~\cite{AviLyu:}. }

\begin{cor}\label{teoprinc2new}
Let $f\colon \ovl{\C} \to \ovl{\C}$ be a rational function of degree $\ge2$ on the Riemann sphere  and let $J=J(f)$ be its Julia set. Assume that we have 
\begin{equation}\label{eq:equalityAvila}
	\DD(f|_J)= \dim_{\rm H}J.
\end{equation}

If $A\subset J$ satisfies $\dim_{\rm H} A <  \dim_{\rm H}{J}$ then
$$
\dim_{\rm H} E_{f|J}^+(A) = \dim_{\rm H} J.
$$
\end{cor}

We obtain an immediate conclusion in the particular case of an expansive map. For that recall that a continuous map $f\colon X \to X$ is {\it expansive} if there exists  $\delta > 0$ such that for each pair of distinct points $x,y\in X$ there is  $n\geq 1$ such that $d(f^n(x),f^n(y))\ge\delta$. 
By the Bowen-Manning-McCluskey formula, in the case of a rational function $f\colon J(f) \to J(f)$ which is expansive equality~\eqref{eq:equalityAvila} holds true (see~\cite[Theorem 3.4]{U}).
Recall that by~\cite[Theorem 4]{DenUrb:91} a rational function of degree $\ge2$ on its Julia set $J(f)$ is expansive (and hence~\eqref{eq:equalityAvila} is true) if, and only if, $J(f)$ does not contain critical points. 

Recent work by Rivera-Letelier and Shen~\cite{RivShe:} establishes~\eqref{eq:equalityAvila} for  a much wider class of  maps. In particular they show that for a rational map of degree $\ge2$ without neutral periodic points, and such that for each critical value $v$ of $f$ in $J(f)$ one has
\[
	\sum_{n=1}^\infty\frac{1}{\lvert (f^n)'(v)\rvert}<\infty
\] 
equalities~\eqref{eq:equalityAvila} hold true (see~\cite[Theorem II and Section 2.1]{RivShe:}) and  Corollary~\ref{teoprinc2new} applies. Note that, in particular,  this is true for Collet-Eckmann maps.

\medskip
Let us compare the main results with other previously known ones. 
Results of this sort have already a long history which  starts with the Jarnik-Besicovitch theorem (see~\cite{Jar:29}) which states that the set of badly approximable numbers\footnote{Recall that a real number $x$ is \emph{badly approximable} if there is a constant $C(x)$ such that for any reduced rational  $p/q$ we have $\lvert p/q -x\rvert>C(x)/q^2$.} in the interval $[0,1]$ is $1$. Observe that $x\in[0,1)$ is \emph{badly approximable} if, and only if, the forward orbit of $x$ relative to the Gauss continued fraction map $f\colon[0,1)\to[0,1)$ does not accumulate at $0$, that is, if $x$ does not belong to the $\{0\}$-exceptional set of points. Here $f$ is  defined by $f(x):=1/x-[1/x]$ if $x>0$, where $[y]$ denotes the integer part of $y$, and $f(0)=0$. This result is then an immediate consequence of the fact that for an expanding Markov map of the interval and any point $x_0$ the $\{x_0\}$-exceptional set has full Hausdorff dimension $1$.

In analogy, in the case of $f$ being an expanding $C^2$ map of a  Riemannian manifold $X$,  it is known that $f$ preserves a probability measure which is equivalent to the Liouville measure~\cite{KrzSzl:69} and hence the set of points whose forward orbit is not dense has zero measure. In particular, for every $x\in X$ the $\{x\}$-exceptional set has zero measure. However, by a result by Urba\'nski~\cite{Urb:91}, this set is large in terms of Hausdorff dimension.  Tseng~\cite{T} 
strengthens this result by showing that in fact this set is a \emph{winning set} in the sense of so-called Schmidt games and hence has full Hausdorff dimension (he also considers the case of a countable set of points $A$).

Abercrombie and Nair \cite{AN2} proved that for a rational map on the Riemann sphere which is uniformly expanding on its Julia set for a given \emph{finite} set of points  $A$ satisfying some additional properties  the $A$-exceptional set has full Hausdorff dimension (see also~\cite{AN1} for a precursor of this work in the case of a Markov map on the interval as well as~\cite{AN3} in a more abstract setting but also requiring uniform expansion of the dynamics and finiteness of the set $A$). Their method of proof (which is similarly used by Dolgopyat~\cite{D} to show Theorem~\ref{Dolgoteoshift} stated below) is based on constructing a certain Borel measure which is supported on the set of points whose forward orbits miss certain neighborhoods of $A$ and then use of a mass distribution principle to estimate dimension. 

Theorems~\ref{teoentropy} and~\ref{teoprinc} and Corollary~\ref{teoprinc2new} generalize these results by Abercrombie and Nair in the sense that we can consider more general sets $A$ and in the sense that we can consider rational maps which are not uniformly expanding. 
They are analogues to \cite[Theorems 1 and 2]{D} by Dolgopyat which allows for more general set $A$ but requires $f$ to be a piecewise uniformly expanding map of the interval.
To the best of our knowledge, our results are the first which apply also in a nonhyperbolic setting. 

Finally, note that there exists a wide range of work on so-called shrinking target problems which are somehow similar -- considering instead of orbits which do not accumulate on a fixed set those orbits which do not hit a neighborhood of a given size which shrinks with the iteration length (see, for example, Hill and Velani~\cite{HV1,HV2,HV3}).

Let us briefly sketch the content of this paper and the idea of the proofs of Theorems \ref{teoentropy} and \ref{teoprinc} (see Section \ref{prova}). 
We will choose a sequence of subsets of $J(f)$ (certain repellers) such that the dynamics inside them is expanding with all their Lyapunov exponents being  close to a given number and their entropy being  close to the entropy of the original system. Such repellers are provided by a construction following ideas of Katok  (see Theorem \ref{Katrin1} in Section~\ref{katoksec}). They have the property that their Lyapunov  exponents and their entropies are close to the ones of an ergodic measure and their Hausdorff dimension is close to the dynamical dimension of the Julia set of whole system. Here we will also invoke the fact~\eqref{eq:dyndimdims}.
Then we will use that (for some iterate of the map) these repellers are conjugate to a subshift of finite type and we will use the following abstract results by Dolgopyat~\cite{D} on shift spaces.

\begin{teorema}[{\cite[Theorem 1]{D}}]\label{Dolgoteoshift}
Let $\sigma\colon  \Sigma^+_M \to \Sigma^+_M$ be a subshift of finite type. 

If $B\subset \Sigma^+_M$ satisfies $h(\sigma,B)< h(\sigma)$, then $h(\sigma,E _{\sigma|\Sigma^+_M}^+(B)) = h(\sigma)$.
\end{teorema}
\noindent
Therefore, Theorem \ref{Dolgoteoshift} guarantees that the entropy on a certain conjugate exceptional set in the subshift coincides with the entropy of the subshift (see Section~\ref{Aset} where general relations for exceptional sets on subsystems are derived). 
To conclude the proof, it is necessary to show a relationship between topological entropy and Hausdorff dimension inside the sub-repellers, which is proved in Section \ref{sec:dimenttt}.

\begin{remark*}{\rm
We remark that the methods in this paper  extend to more general conformal $C^{1+\alpha}$ maps $f$ of a Riemannian manifold $X$ and a compact invariant subset $W\subset X$ studying  exceptional sets in $W$ (relative to the dynamics of $f|_W$). We point out that  one essential ingredient is the  equality%
\footnote{Note that in such a context  we always have $\hD(f|_W)\le \DD(f|_W)\le \dim_{\rm H}W$. Indeed, it suffices to observe that to each  conformal expanding repeller $Y$ there exists an ergodic measure $\mu$ of maximal dimension $\dim_{\rm H}\mu=\dim_{\rm H}Y$ (e.g.~\cite[Theorem 1]{GP}). This implies the first inequality, the second one is immediate.} 
 between hyperbolic dimension and dynamical dimension of $f|_W$ (as in~\eqref{eq:dyndimdims}).
Another one is the possibility to approach any ergodic measure with positive entropy and positive Lyapunov exponent by a certain repeller (see Theorem~\ref{Katrin1}). Then a key point is to guarantee that such repellers are contained in $W$. Whenever these facts were true, our proofs extend to such a map and Theorems~\ref{teoentropy} and~\ref{teoprinc} (and Corollary~\ref{teoprinc2new} in case one has equality between  dynamical and Hausdorff dimension as in~\eqref{eq:equalityAvila})  continue to hold true exchanging the Julia set for $W$.
 
For example, in~\cite{RivShe:} the authors consider  the Julia set of a certain  $C^3$ multimodal interval map with nonflat critical points and without neutral periodic points.  
We refrain from giving all the details and refer to~\cite{RivShe:} for the precise definitions.
Under additional conditions, in particular on the critical points, they establish the corresponding equalities~\eqref{eq:equalityAvila} for such maps. The above results apply in this context (see also~\cite{Cam:15}).	
}
\end{remark*}

\section{Dimension and  entropy of a $(\chi,\epsilon)$-repeller} \label{sec:dimenttt}

In this section we will derive a relationship between the Hausdorff dimension and the topological entropy for a specific type of repellers that we call $(\chi,\epsilon)$-repellers.
First, we recall briefly dimension and entropy and some of their properties.

\subsection{Hausdorff Dimension}\label{HD}

Let $X$ be a metric space. Given a set $Y\subset X$ and a nonnegative number $d \in\R$, we denote the {\it $d$-dimensional Hausdorff measure} of $Y$ by   
$$
\mathcal{H}^d(Y):=
\lim_{r \to 0}\mathcal{H}_{r}^d(Y),
$$
 where 
 $$
 	{\mathcal H}^d_r(Y)
	:= \inf\left\{\displaystyle\sum_{i=1}^{\infty}(\diam U_i)^d\colon  
		Y \subset \bigcup_{i=1}^\infty U_i, \diam U_i <r\right\} ,
$$
where $\diam U_i$ denotes the diameter of $U_i$.
Observe that $\mathcal{H}^d(Y)$ is monotone  nonincreasing in $d$. 
Furthermore, if $d\in(a,b)$ and $\mathcal{H}^d(Y)<\infty$ 
then $\mathcal{H}^b(Y)=0$ 
and $\mathcal{H}^a(Y)=\infty.$
The unique value $d_0$ at which $d\mapsto \mathcal{H}^d(Y)$ 
jumps from $\infty$ to $0$ is the {\it Hausdorff dimension} of $Y$, that is,  
$$
\dim_{\rm H} Y=\inf\{d\geq 0 \colon {\mathcal H}^d(Y)=0\}=
\sup\{d\geq 0 \colon {\mathcal H}^d(Y)=\infty\}.
$$

We recall some of its properties:
\begin{itemize}
\item [(H1)] Hausdorff dimension is monotone: if $Y_1\subset Y_2\subset X$ then $\dim_{\rm H} Y_1\leq\dim_{\rm H} Y_2$. 
\item [(H2)] Hausdorff dimension is countably stable: $\dim_{\rm H}\bigcup_{i=1}^\infty B_i=\sup_i\dim_{\rm H}B_i$. 
\end{itemize}

\subsection{Topological Entropy}\label{Ent}
Let us now define topological entropy. We will follow the more general approach by Bowen \cite{B} considering the topological entropy of a general  (i.e., not necessarily compact and invariant) set.  

Let $X$ be a compact metric space. Consider a continuous map $f\colon X\to X$, a set $Y\subset X$,  and a finite open cover $\mathscr{A} = \{A_1, A_2,\ldots, A_n\}$ of $X$. Given $U\subset X$ we write $U \prec \mathscr{A}$ if there is an index $j$ so that $U\subset A_j$, and $U\nprec\mathscr{A}$ otherwise. 
Taking $U\subset X$ we define 
$$
	n_{f,\mathscr{A}}(U) := 
		\begin{cases}
		0&\text{ if } U \nprec \mathscr{A},\\
		\infty &\text{ if } f^k(U)\prec \mathscr{A}\,\,\forall k\in\mathbb{N},\\
		 \ell&\text{ if }  f^k(U)\prec \mathscr{A}\,\, \forall k\in \{0, \dots,  \ell-1\},f^\ell(U)\nprec\mathcal{A}.
		\end{cases}
$$
If $\mathcal U$ is a countable collection of open sets, given $d>0$ let
\[
	 m(\mathscr A,d,\mathcal U)
	:= \sum_{U\in\mathcal U}e^{-d \,n_{f,\mathscr{A}}(U)}.
\]
Given a set $Y\subset X$, let 
$$
	m_{\mathscr{A}, d} (Y) 
	:= \lim_{\rho \to 0}\inf \Big\{m(\mathscr A,d,\mathcal U)\colon
		Y \subset\disp\bigcup_{U\in\mathcal U} U, e^{-n_{f,\mathcal{A}}(U)}<\rho
		\text{ for every } U\in\mathcal U
	\Big\}.
$$
Analogously to the Hausdorff measure, $d\mapsto m_{\mathcal{A},d}(Y)$ 
jumps from $\infty$ to $0$ at a unique critical point and we define
$$
  h_{\mathscr{A}}(f,Y) 
	:= \inf\{d\colon m_{\mathscr{A}, d}(Y)=0\}
   = \sup\{d\colon m_{\mathscr{A}, d}(Y)=\infty\}. 
$$
The \emph{topological entropy} of $f$ on the set $Y$ is defined by 
$$
	h(f,Y) 
	:= \sup_{\mathscr{A}} h_{\mathscr{A}}(f,Y) ,
$$
Observe that for any finite open cover $\mathscr{A}$ of $Y$ there exists another finite open cover $\mathscr{A}'$ of $Y$ with smaller diameter such that $h_{\mathscr{A}'}(f,Y) \geq h_{\mathscr{A}}(f,Y)$, which means that, in fact, for any $R>0$
$$ 
	h(f,Y)
	= \sup\{h_{\mathscr{A}}(f,Y)\colon 
		\mathscr{A} \texto{ finite open cover of } Y,\diam{\mathscr{A}}< R\}.
$$
When $Y=X$, we simply write $h(f) = h(f,X)$. To avoid confusion, we sometimes explicitly write $h(f|_X,Y)=h(f,Y)$.

In~\cite[Proposition 1]{B}, it is shown that  in the case of a compact set $Y$ this definition is equivalent to the canonical definition of topological entropy (see, for example, \cite[Chapter 7]{W}). 

We recall some of its properties which are relevant in our context (see~\cite{B}). 
\begin{itemize}
\item[(E1)] Conjugation  preserves entropy: If $f\colon X\to X$ and $g\colon Z \to Z$ are topologically conjugate, that is, there is a homeomorphism $\pi\colon  X \to Z$ with $\pi \circ f = g\circ \pi$, then $h(f,Y) = h(g,\pi(Y))$ for every $Y\subset X$. 
\item[(E2)] Entropy is invariant under iteration: $h(f,f(Y)) =  h(f,Y)$. 
\item[(E3)] Entropy is countably stable: $h(f,\bigcup_{i=1}^\infty B_i ) = \sup_i h(f,B_i).$
\item[(E4)]  $h(f^m,Y) = m\cdot h(f,Y)$ for all $m\in\N$.
\item[(E5)] Entropy is monotone: if $Y\subset Z\subset X$ then $h(f,Y)\le h(f,Z)$.
\end{itemize}

\subsection{$(\chi,\epsilon)$-repellers} \label{chirepellers}

In this section let $X$ be a  Riemannian manifold and $f\colon X \to X$ be a differentiable map. We call $f$ {\it conformal} if for each $x\in X$ we have $D_xf = a(x)\cdot\Isom_x$, where $a(x)$ is a positive scalar and $\Isom_x\colon T_xX\to T_{f(x)}X$ is an isometry; in this case we simply write $ a(x) = \lvert f'(x)\rvert$.
 We say that a set $W\subset X$ is \emph{forward invariant} if $f(W)= W$. A compact set $W\subset X$ is said to be \emph{isolated} if there is an open neighborhood $V$ of $W$ such that $f^n(x)\in V$ for all $n\ge0$ implies $x\in W$.
Given a $f$-forward invariant subset $W\subset X$ we call $f|_W$ {\it expanding} if there exists $n\geq 1$ such that, for all $x \in W$ we have   
$$
|(f^n)'(x)|>1. 
$$    

\begin{defi}
A compact $f$-forward invariant isolated expanding set $W\subset X$ of a conformal map $f$ is said to be a \emph{conformal expanding repeller}.

Given numbers $\chi>0$ and $\epsilon\in(0,\chi)$, we call a conformal expanding repeller $W\subset X$ a \emph{$(\chi, \epsilon)$-repeller} if for every $x\in W$ we have
\begin{equation} \label{proprichirepeller}
	\limsup_{n\to \infty} \Big\lvert\frac{1}{n} \log \,\lvert (f^n)'(x)\rvert - \chi\Big\rvert
	< \epsilon.
\end{equation}
\end{defi}

In the following, we will collect some important estimates between Hausdorff dimension and topological entropy of $(\chi, \epsilon)$-repellers. 
The following estimate is of similar spirit as~\cite[Lemma 2]{D}. The method of proof is partially inspired by~\cite{M} and \cite[proof of Theorem 1.2]{BG}. 
See also~\cite{MaWu:10} for similar results.
We will first prove a general result and then consider the particular case of $(\chi, \epsilon)$-repellers.

\begin{prop}\label{proplema 2.0.1} 
Consider  a  Riemannian manifold  $X$ and $f\colon X \to X$ a conformal $C^{1+\alpha}$ map. Let  $W \subset X$ be a conformal expanding repeller.
Let $Z\subset W$ and let $\chi>0$ and $\epsilon\in(0,\chi)$ be numbers such that for every $x\in Z$ we have~\eqref{proprichirepeller}.

Then we have
$$
	\frac{h(f|_W,Z)}{\chi + \epsilon} 
	\leq \dim_{\rm H}Z 
	\leq \frac{h(f|_W,Z)}{\chi - \epsilon} .
$$
\end{prop}

\begin{proof}
In what follows, in order to simplify notations we avoid conceptually unnecessary use of coordinate charts. 

Given $N\in\mathbb N$, we define the following level sets: 
$$
	Z_N 
	:= \Big\{ x \in Z\colon  
		\Big\lvert \frac{1}{n}\log\,\lvert(f^n)'(x)\rvert - \chi\Big\rvert<\epsilon 
		\text{ for all } n\geq N\Big\}.
$$
By hypothesis on $Z$, we have that 
\begin{equation}\label{Zunion}
Z=\disp\bigcup_{N\in\N} Z_N.
\end{equation}
Observe that $Z_N\subset Z_{N'}$ for $N<N'$. 
Given $N\in\N$, for all  $x\in Z_N$ and all $k\geq N$ we have 
\begin{equation}\label{lema2cotader}
	e^{k(\chi - \epsilon)} 
	< |(f^{k})'(x)| 
	< e^{k(\chi + \epsilon)} .
\end{equation}

On a sufficiently small  neighborhood $V$ of $W$  we have $|f'|\ne0$ and hence for $\theta >0$ there exists $R =R(\theta)> 0$ such that if $z_1, z_2 \in V$ and $d(z_1, z_2) < R$ then  
\begin{equation}\label{lemdesigeqlog1}
	\big\lvert\log\,\lvert f'(z_1)\rvert - \log\,\lvert f'(z_2)\rvert\big\rvert < \theta. 
 \end{equation}

\smallskip\noindent\textbf{Step 1:} 
We start by showing 
\begin{equation}\label{eq:dimonesid}
	h(f|_W,Z)\le (\chi+\epsilon)\dim_{\rm H}Z.
\end{equation}

Fix $N\in\N$. 
Fix some $\theta>0$ and let $R=R(\theta)$ as above. 

We start by estimating the entropy  on $Z_N$. For that we choose some finite open cover $\mathscr{A}$ of $W$ with  $\diam\mathscr{A} \leq R$. Let $\ell=\ell(\mathscr A)$ denote a Lebesgue number of $\mathscr A$. Let 
\[
	r_0=r_0(N)
	:=\ell\min_{0\leq k \leq N} \min_{x\in \ovl{V}}\,\lvert(f^k)'(x)\rvert^{-1}. 
\]

We prove the following intermediate result.

\begin{claim}\label{claimmmm}
	For every $\gamma>(\chi+\epsilon+\theta)\dim_{\rm H}Z_N$, we have  $m_{\mathscr A, \gamma} (Z_N) =0$.
\end{claim}	

\begin{proof}
Let $D:=\gamma/(\chi +\epsilon+\theta)$.
Let $c=\log(\ell/2)/(\chi+\epsilon+\theta)$.
Let $\zeta>0$. 
As $D>\dim_{\rm H} Z_N$, there is $\rho_0>0$ such that for all $r\in(0,\rho_0)$ we have that 
\[
	\inf\Big\{\sum_ir_i^D\colon Z_N\subset\bigcup_iB(x_i,r_i),r_i<r\Big\}<
	\zeta e^{c\gamma}.
\]
Let $\rho_1:=\min\{r_0, \rho_0\}$. 
Then, for every $\rho \in (0,\rho_1)$ there is $r\in (0, \rho)$  also satisfying
$$
	r
	< (e^c \rho)^{\chi +\epsilon+\theta}
$$
and a cover $\mathcal U=\{U_i\}$ of $Z_N$ by open balls $U_i=B(x_i,r_i)$, $r_i<r$, so that
\begin{equation}\label{eq:22222}
	\sum_ir_i^D<
	\zeta e^{c\gamma}.
\end{equation}

For every $U_i\in\mathcal{U}$, for any  $z_1,z_2 \in U$ for all $j \in\{0,\ldots,n_{f,\mathscr{A}}(U_i) - 1\}$ we have 
\[
	d(f^j(z_1), f^j(z_2)) <\diam\mathscr A\le R.
\]	
From~\eqref{lemdesigeqlog1} it follows that  for every $k=1,\ldots,n_{f,\mathscr{A}}(U_i)$ we have
\[
	\big|\log|(f^k)'(z_1)| - \log|(f^k)'(z_2)|\big| 
	\leq \sum_{j=0}^{k-1}\big |\log|f'(f^j (z_1))| - \log|f'(f^j(z_2))|\big| 
	\leq k \theta
\]
and hence 
\begin{equation}\label{lemadeslogeq2}
e^{-k\theta} \leq \disp\frac{|(f^k)'(z_1)|}{|(f^k)'(z_2)|} \leq e^{k\theta} .
\end{equation}

Given $i$, for $x\in U_i\in\mathcal U$ let $F(x)=f^{n_{f,\mathscr{A}}(U_i)}(x)$. By definition of $n_{f,\mathscr{A}}(U_i)$ and of the Lebesgue number  $\ell$, for every $U_i\in\mathcal U$ it follows that $\ell \leq \diam f^{n_{f,\mathscr{A}}(U_i)}(U_i)=\diam F(U_i)$. Consider $x,y\in \overline{U_i}$ such that $\diam F(U_i) = d(F(x), F(y))$. 
Consider the shortest path $\gamma\colon [0,1] \to X$ linking $x$ to $y$, which is completely contained in $\overline{U_i}$ since $\mathcal U$ is a cover by balls.
Thus  
$$
\ell  \leq  d(F(x) , F(y)) 
\leq \int^1_0 |(F\circ \gamma)'(t)|\,dt 
 =  \int^1_0 |F'(\gamma(t))||\gamma ' (t)|\,dt. 
$$
Observe that $r_i < r_0$ implies that $n_{f, \mathcal{A}}(U_i) >N$. 
Considering $z\in U_i\cap Z_N$, with 
$k = n_{f,\mathscr{A}}(U_i)>N$ in~\eqref{lemadeslogeq2} and~\eqref{lema2cotader} we conclude
\[\begin{split}
	\ell & \leq \int^1_0 \disp\frac{|F'(\gamma(t))|}{|F'(z)|} |F'(z)|\,|\gamma ' (t)|\,dt\\
        \text{by~\eqref{lemadeslogeq2}}\quad
         &\leq  e^{ n_{f,\mathscr{A}}(U_i)\theta}\int^1_0 
        		|(f^{n_{f,\mathscr{A}}(U_i)})'(z)|\,|\gamma ' (t)|\,dt\\ 
       \text{by~\eqref{lema2cotader}}\quad
       & < e^{n_{f,\mathscr{A}}(U_i)\theta}
       		e^{n_{f,\mathscr{A}}(U_i)(\chi + \epsilon)} \diam U_i.
\end{split}\]
Recalling the definition of $c$ we obtain  
\begin{equation}\label{lemainteqp1}
	e^{- n_{f,\mathscr{A}}(U_i)} 
	< \big(\ell^{-1}\diam U_i\big)^{1/(\chi+\epsilon+\theta)}
	= e^{-c} (\frac12\diam U_i)^{1/(\chi + \epsilon+\theta)}.
\end{equation}
Since $\diam{U_i} < 2r < 2(e^c \rho)^{\chi +\epsilon+\theta}$ we have  
$	e^{-n_{f, \mathcal{A}}(U_i)}  
	< \rho. 
$

Then,  we have
\[\begin{split}
	m(\mathscr A,\gamma,\mathcal U)
	&= \sum_{U_i\in\mathcal U}e^{-\gamma\, n_{f,\mathscr A}(U_i)}\\
	\text{by~\eqref{lemainteqp1} }\quad
	&\le  e^{-c\gamma}\sum_{U_i\in\mathcal U} 
		(\frac12\diam U_i)^{\gamma/(\chi + \epsilon+\theta)}
	= e^{-c\gamma}\sum_{U_i\in\mathcal U} 
		r_i^D
	\\
	\text{by~\eqref{eq:22222}}\quad 	
	&< e^{-c\gamma} \zeta e^{c\gamma} = \zeta.
\end{split}\]

Summarizing, for arbitrary $\zeta>0$, there exists $\rho_1>0$ such that for any $\rho\in( 0,\rho_1)$ we can cover $Z_N$ by a family of balls $U_i$ satisfying $e^{-n_{f, \mathcal{A}}(U_i)} < \rho$ and $\sum_{U_i\in \mathcal{U}} e^{-\gamma n_{f,\mathcal{A}}(U_i)} < \zeta$.   
Thus $m_{\mathscr{A}, \gamma}(Z_N) =0$ as claimed.
\end{proof}

By Claim~\ref{claimmmm}, for every $\gamma>(\chi+\epsilon+\theta)\dim_{\rm H}Z_N$, we have $m_{\mathscr A,\gamma}(Z_N)=0$, which implies  $h_{\mathscr{A}}(f,Z_N)\leq\gamma$. 
Since $\gamma>(\chi+\epsilon+\theta)\dim_{\rm H}Z_N$ is arbitrary,  therefore
$$
	h_{\mathscr A}(f,Z_N)\leq (\chi+\epsilon+\theta)\dim_{\rm H} Z_N.
$$
Thus, as $\mathscr A$ was arbitrary (but sufficiently small)
\[
	h(f|_W,Z_N)\le (\chi+\epsilon+\theta)\dim_{\rm H} Z_N.
\]	 
Since $\theta$ was arbitrary, we obtain
\[
	h(f|_W,Z_N)\le (\chi+\epsilon)\dim_{\rm H}Z_N.
\]

Now recall that $N\ge1$ was arbitrary.	
With~\eqref{Zunion} and countable stabilities (H2) of  Hausdorff dimension and (E3) of entropy we conclude~\eqref{eq:dimonesid} from 
$$
	\dim_{\rm H} Z 
	= \sup_N \dim_{\rm H} Z_N 
	\geq \sup_N \disp\frac{h(f|_W, Z_N)}{\chi+\epsilon} 
	= \frac{1}{\chi+\epsilon}\sup_Nh(f|_W,Z_N)
	= \frac{1}{\chi+\epsilon}h(f|_W, Z).
$$
This concludes Step 1.

\smallskip\noindent
\textbf{Step 2:}  
We now show
\begin{equation}\label{eq:upperboudim}
		\dim_{\rm H}Z \le \frac{h(f|_W,Z)}{\chi-\epsilon}.
\end{equation}

Fix some $N\in\mathbb N$. 
Fix some $\theta\in(0,\chi-\epsilon)$ and let $R=R(\theta)$ as above.

We start by estimating the dimension of $Z_N$.
Fix some $\tau>0$ and 
denote $D:=(h(f|_W,Z_N)+\tau)/(\chi-\epsilon-\theta)$. 
Observe that 
\[
	(\chi - \epsilon -\theta)D 
	= h(f|_W,Z_N)+\tau >  h(f|_W,Z_N) 
	= \sup_{\mathscr{A}} h_{\mathscr{A}} (Z_N).
\]	 
Hence, for any finite open cover $\mathscr{A}$ of $W$ we have $m_{\mathscr A,(\chi - \epsilon -\theta)D}(Z_N) = 0$. 
Choose some cover $\mathscr{A}$ with $\diam\mathscr A\le R$.  

Given some $U\prec \mathscr{A}$ with $n=n_{f,\mathscr A}(U)<\infty$, fix some point $x\in U\cap Z_N$ and consider the sequence $x_k=f^k(x)$, $k=0,\ldots,n-1$.
So for each $k$ there is some $A_k\in\mathscr A$ with $x_k\in f^k(U)\subset A_k$.
Denote by $f^{-k}_{x_{n-1-k}}$ the inverse branch $g$ of $f^k$ so that $(g\circ f^k)(x_{n-1-k})=x_{n-1-k}$. We observe the following preliminary fact.

\begin{claim}\label{cla:fuenf}
	For every $k=0,\ldots,n-1$  for every $x\in U$ we have
	\[
		\diam f^{-k}_{x_{n-1-k}} (f^{n-1}(U))
		\le \lvert (f^k)'(x_{n-1-k})\rvert^{-1} e^{k\theta}\cdot R.
	\]	
\end{claim}

\begin{proof}
The proof is by induction. For $k=0$ we have $f^{n-1}(U)\subset A_{n-1}\in\mathscr A$ and hence $\diam f^{n-1}(U)\le R$. For $k\in \{1,\ldots,n-1\}$, suppose the claim holds for $k-1=j$. 
Let $V_{j+1}:=f^{-(j+1)}_{x_{n-1-(j+1)}}(f^{n-1}(U))= f^{-1}_{x_{n-1-(j+1)}}(V_j)$ and observe that, in particular, $V_{j+1}\subset A_{j+1}\in\mathscr A$. Since for every $y,z\in A_{j+1}$, using~\eqref{lemdesigeqlog1},  we have that $\lvert f'(y)\rvert/\lvert f'(z)\rvert\le e^\theta$, we can conclude 
\[
	\diam V_{j+1}
	\le \sup_{y\in V_{j+1}}\lvert f'(y)\rvert^{-1}\diam V_j
	\le \lvert f'(x_j)\rvert^{-1}e^\theta\diam V_j.
\]
Invoking the induction hypothesis for $k =j$, we obtain
\[\begin{split}
	\diam V_{j+1}
	&\le \lvert f'(x_j)\rvert^{-1}e^\theta 
		\cdot\lvert (f^j)'(x_{n-1-j})\rvert^{-1}e^{j\theta} \cdot R\\
	&= \lvert (f^{j+1})'(x_{n-1-(j+1)})\rvert^{-1}e^{(j+1)\theta}\cdot R,	
\end{split}\]
proving the assertion for $j+1$. This proves the claim.
\end{proof}

\begin{claim}\label{cla:clavier}
	$\mathcal{H}^D(Z_N) = 0$. 
\end{claim}

\begin{proof}
Let $\eta>0$. 
Observe that $m_{\mathscr A,(\chi - \epsilon -\theta)D}(Z_N)=0$ implies that there is $\rho_0>0$ such that for every $\rho \in(0,\rho_0)$ we have that
\[
	 \inf\Big\{ m(\mathscr A,(\chi-\epsilon-\theta)D,\mathcal U)
		\colon Z_N\subset\bigcup_{U\in\mathcal U}U, e^{-n_{f,\mathscr A}(U)} < \rho\Big\}
	< \eta e^{-(\chi - \epsilon -\theta)D} R^{-D}.
\]
Consider $r_1 < \min\{\rho_0, e^{-(N+1)}\}$. 
Then, for every $r \in (0,r_1)$ there is $\rho \in (0,r)$ also satisfying
\begin{equation}\label{eq:definitionr11}
	e^{\chi-\epsilon-\theta}R\cdot \rho^{\chi-\epsilon-\theta}
	< r.
\end{equation}
Hence,  there exists a cover $\mathcal{U}=\{U_i\}$ of $Z_N$ satisfying $e^{-n_{f, \mathcal{A}}(U_i)} < \rho$ and
\begin{equation}\label{eq:bedingungentr2}
	 m(\mathscr A, (\chi - \epsilon -\theta)D,\mathcal U) 
	 < \eta  e^{-(\chi - \epsilon -\theta)D} R^{-D}. 
\end{equation}

Note that $\rho< e^{-(N+1)}$ implies $n_{f,\mathscr A}(U_i)>N+1$. Also recall that  $f^k(U_i)$ lies inside an element of $\mathscr A$ for every $k=0,\ldots,n_{f,\mathscr A}(U_i)-1$. 
Consequently, with Claim~\ref{cla:fuenf} for $k=n_{f,\mathscr A}(U_i)-1$ and $x\in  Z_N \cap U_i$ we obtain
\[
	\diam U_i
	\le \lvert (f^{n_{f,\mathscr A}(U_i)-1})'(x)\rvert^{-1} 
		e^{(n_{f,\mathscr A}(U_i)-1)\theta}\cdot R
	< e^{-(n_{f,\mathscr A}(U_i)-1)(\chi-\epsilon-\theta)}\cdot R.	
\]
Thus, since $e^{-n_{f,\mathscr A}(U_i)}< \rho$, we have that 
\[\begin{split}
	\diam U_i
	&< e^{\chi-\epsilon-\theta}
		R\cdot e^{-n_{f,\mathscr A}(U_i)(\chi-\epsilon-\theta)}
	<e^{\chi-\epsilon-\theta}
		R\cdot\rho^{\chi-\epsilon-\theta}\\
	\text{by~\eqref{eq:definitionr11}}\quad
	&<r. 
\end{split}\]
By~\eqref{eq:bedingungentr2} and above inequality, 
\[\begin{split}
	\sum_{U_i\in\mathcal U}(\diam U_i)^D
	&\le \sum_i\left(e^{\chi-\epsilon-\theta}R\cdot e^{-n_{f,\mathscr A}(U_i)(\chi-\epsilon-\theta)}\right)^{D}\\
	&= e^{(\chi-\epsilon-\theta)D}R^D
		\cdot m(\mathscr A,D(\chi-\epsilon-\theta),\mathcal U)	
	< \eta.
\end{split}\]

Summarizing, for arbitrary $\eta>0$, there exists $r_1>0$ such that for every $r\in(0,r_1)$ we can cover $Z_N$ by $\mathcal{U}$ such that $\diam U_i< r$ for every $U_i\in\mathcal U$ and $\sum_{U_i\in \mathcal{U}} (\diam U_i)^D < \eta$.   
Thus, $\mathcal H^D(Z_N)=0$, proving the claim.
\end{proof}

Claim~\ref{cla:clavier} now implies immediately
\[
	\dim_{\rm H}Z_N\le\frac{h(f|_W,Z_N)+\tau}{\chi-\epsilon-\theta}.
\]	
As $\tau>0$ and $\theta\in(0,\chi-\epsilon)$ were arbitrary, we conclude  
$$
\dim_{\rm H} Z_N \leq \disp\frac{h(f|_W, Z_N)}{\chi - \epsilon}.
$$
Finally, recall that $N$ was arbitrary, by~\eqref{Zunion}, (E3), and (H2), we obtain
$$
	\dim_{\rm H} Z 
	= \sup_N \dim_{\rm H} Z_N 
	\leq \sup_N \frac{h(f|_W, Z_N)}{\chi - \epsilon} 
	= \frac{h(f|_W,Z)}{\chi - \epsilon}.
$$
This shows~\eqref{eq:upperboudim} and finishes the proof of  the proposition.
\end{proof}

The following is now an immediate consequence of  Proposition~\ref{proplema 2.0.1}.

\begin{cor}\label{proplema 2.0cor} 
Consider  a  Riemannian manifold $X$ and $f\colon X \to X$ a conformal $C^{1+\alpha}$ map. Let  $W \subset X$ be a  $(\chi,\epsilon)$-repeller.

Then for every $Z\subset W$ we have
$$
	\frac{h(f|_W,Z)}{\chi + \epsilon} 
	\leq \dim_{\rm H}Z 
	\leq \frac{h(f|_W,Z)}{\chi - \epsilon} .
$$
\end{cor}

Finally, we provide some further consequences which we will need in the sequel. Given $N\in \N$ let $R\subset W$ be a compact set satisfying
\begin{equation}\label{eq:W}
	f^N(R)=R \quad\text{ and }\quad W=\bigcup_{i=0}^{N-1}f^i(R).
\end{equation}

\begin{lema}\label{lem:simple}
	$h(f|_W)=\frac1Nh(f^N|_R)$.
\end{lema}

\begin{proof}
	By (E3), (E2), (E4) and the $f^N$-invariance of $R$ we have 
\[	
	h(f|_W)
	=\max_ih(f|_W,f^i(R))
	=h(f|_W,R)
	=\frac1Nh(f^N|_W,R)
	=\frac1Nh(f^N|_R).
\]	
This proves the lemma.\end{proof}

\begin{lema}\label{afirm2.1.1}
Consider  a  Riemannian manifold $X$ and $f\colon X \to X$ a conformal $C^{1+\alpha}$ map. 
Suppose that $W \subset X$ is a  $(\chi,\epsilon)$-repeller of positive entropy
and $R\subset W$  a compact set satisfying $f^N(R)=R$ and $W=\bigcup_{i=0}^{N-1}f^i(R)$  for some $N\ge1$. 

Then for every $Y\subset R$ we have   
\[
	\dim_{\rm H} Y 
	\geq 
      \frac{h(f|_{W},Y)}{h(f|_{W})} 
		\frac{(\chi - \epsilon)}{(\chi + \epsilon)}
		\dim_{\rm H} W	
    = \frac{h(f^N|_{R},Y)}{h(f^N|_{R})} 
		\frac{(\chi - \epsilon)}{(\chi + \epsilon)}
		\dim_{\rm H} W.
\]
\end{lema}

\begin{proof}
Applying  Corollary~\ref{proplema 2.0cor} we have
\[
	\frac{1}{h(f|_W)}(\chi-\epsilon)\dim_{\rm H} W \le 1.
\]
Given $Y\subset R\subset W$, we also have
\[
	 \frac{h(f|_W,Y)}{\chi+\epsilon}\le \dim_{\rm H} Y 
\]
Multiplying both  inequalities, we obtain the first inequality. 
The equality is a consequence of  Lemma \ref{lem:simple}, the Property (E4) and the $f^N$- invariance of $R$. 
\end{proof}

\section{Expanding repellers for nonuniformly expanding maps}\label{katoksec}
In order to find an approximation of ergodic quantifiers of the -- in general non-expanding -- maps,  we  follow an idea by Katok to construct suitable repellers. For a proof of the  following theorem see~\cite[Chapter 11.6]{PU} and ~\cite[Theorems 1 and 3]{G}.    

\begin{teorema}\label{Katrin1}
Consider  a  Riemannian manifold $X$ and $f\colon X \to X$ a conformal $C^{1+\alpha}$ map. Let $\mu$ be an   $f$-invariant ergodic measure with positive entropy $h_\mu(f)$ and positive Lyapunov exponent 
\[
	\chi(\mu) := \int\log\,\lvert f'\rvert\,d\mu.
\]	 
 
Then for all $\epsilon>0$ there is a compact set $W_{\epsilon} \subset X$ such that $f|_{W_{\epsilon}}$ is a conformal expanding repeller satisfying:
\begin{itemize}
\item[(a)] 
	$h_{\mu} (f) +\epsilon\geq h(f|_{W_{\epsilon}}) \geq h_{\mu} (f) -\epsilon$,\\[-0.4cm]
\item[(b)] For every $f$-invariant ergodic measure $\nu$ supported in $W_\epsilon$ we have
\[
	\big\lvert \chi(\nu)-\chi(\mu)\big\rvert <\epsilon.
\] 
\end{itemize}
In particular, $W_\epsilon$ is a $(\chi(\mu),\epsilon)$-repeller.

Moreover, there is a compact set $R_\epsilon\subset W_\epsilon$ and some $N=N(\epsilon)\in\N$ such that $f^N(R_\epsilon) = R_\epsilon$, $f^N|_{R_{\epsilon}}$ is expanding and topologically conjugate to a topologically mixing subshift of finite type, and we have
$$
W_\epsilon = \disp\bigcup_{i=0}^{N-1} f^i(R_\epsilon). 
$$ 
\end{teorema}

These repellers $W_\epsilon$ have good dimension properties as we shall see below.  In particular, we can apply 
Corollary~\ref{proplema 2.0cor} to them. 

For the following result recall the definition of the dynamical dimension in~\eqref{def:DD}.

\begin{lema}\label{lemaqeaprox} 
	Let $f\colon \ovl{\C} \to \ovl{\C}$ be a rational function of degree $\ge2$ on the Riemann sphere and let $J=J(f)$ be its Julia set. 

Then there exist a sequence of probability measures $(\mu_n)_n$ and a sequence of positive numbers $(\epsilon_n)_n$ with $\lim_{n\to 0}\epsilon_n = 0$ and $\epsilon_n<\chi(\mu_n)/n$ such that there are $(\chi(\mu_n), \epsilon_n)$-repellers  $W_n = W_n(\mu_n)\subset J$ satisfying  
\[
\lim_{n\to \infty}  \dim_{\rm H} W_n = \DD(f|_{J}).
\]
\end{lema}

\begin{proof}
First note that for a $f$-invariant ergodic probability measure $\mu$ of a rational function with positive Lyapunov exponent $\chi(\mu)$ we have 
\begin{equation}\label{eq:Mane}
	\dim_{\rm H}\mu = \frac{h_{\mu}(f)}{\chi(\mu)}
\end{equation}
(\cite{Man:88}, see also~\cite[Chapters 8--10]{PU}).

Given $n\in \N$, by definition of the dynamical dimension, there is an $f$-ergodic probability measure $\mu_n$  with positive entropy (and hence positive Lyapunov exponent) such that 
\begin{equation}\label{lemadihqe}
	\dim_{\rm H}\mu_n \geq \DD(f|_{J}) - \frac{1}{n} .
\end{equation}
Choose $\epsilon_n>0$ satisfying $\epsilon_n<\chi(\mu_n)/n$.
Let $W_n$ be a $(\chi(\mu_n),
\epsilon_n)$-repeller provided by Theorem \ref{Katrin1} applied to $\mu_n$ and recall that there is $N=N(\epsilon_n)$ and $R_n\subset W_n$ such that $f^N|_{R_n}$ is expanding and conjugate to a mixing subshift of finite type. Observe that $\dim_{\rm H}W_n=\dim_{\rm H}R_n$. 
Also observe that $W_n\subset J$.
Applying Bowen's formula (see~\cite{GP}) for $f^N|_{R_n}$, with $s_n=\dim_{\rm H}R_n$ we have
\[
	0=\sup_\nu\big(h_\nu(f^N)- s_nN\chi(\nu)\big),
\]
where the supremum is taken over all $f^N$-invariant  measures $\nu$ supported in $R_n$.
Recall that for every invariant measure $\nu$ for $f^N\colon R_n\to R_n$ we get an invariant measure $\hat\nu$ for $f\colon W_n\to W_n$ by defining $\hat\nu:=\frac1N(\nu+f_\ast\nu+\ldots+f^{N-1}_\ast\nu)$ and observe that $h_\nu(f^N)=Nh_{\hat\nu}(f)$. Further, $h(f^N|_{R_n})=Nh(f|_{W_n})$ (Lemma~\ref{lem:simple}).
By the variational principle for topological entropy (see e.g. \cite[Chapter 9]{PU}), we can take $\nu$ such that 
$h_\nu(f^N) \ge Nh(f|_{W_n}) -N\epsilon_n$, which implies 
$$
0\ge h(f|_{W_n}) - \epsilon_n -s_n\chi(\nu). 
$$
From Theorem~\ref{Katrin1} 
we obtain
\[
	0 
	\ge  h_{\mu_n}(f) -2\epsilon_n - s_n (\chi(\mu_n)+\epsilon_n),
\] 
which implies
\[
	s_n\ge \frac{h_{\mu_n}(f) -2\epsilon_n}{\chi(\mu_n)+\epsilon_n}.
\]
Hence, by~\eqref{eq:Mane}, we conclude 
\begin{equation} \label{dimmu}
	s_n 
	\ge \dim_{\rm H}\mu_n\left(\frac{\chi(\mu_n)}{\chi(\mu_n) + \epsilon_n}\right) 
		- \frac{2\epsilon_n}{\chi(\mu_n) + \epsilon_n}. 
\end{equation}
As we required $0<\epsilon_n < \chi(\mu_n)/n$, inequalities (\ref{lemadihqe}) and (\ref{dimmu}) show that 
\[
\left(\DD(f|_{J}) - \frac{1}{n}\right) \frac{1}{1+1/n} - \frac{2}{n+1}
\leq s_n=\dim_{\rm H} W_n. 
\]
Finally, it follows definition of hyperbolic dimension and ~\eqref{eq:dyndimdims} that
\[
	\dim_{\rm H} W_n
	\le \hD(f|_{J})
	= \DD(f|_{J}).
\]
Taking the limit when $n\to\infty$, we obtain 
$$
\disp\lim_{n\to \infty} \dim_{\rm H} W_n = \DD(f|_{J}). 
$$
This proves the lemma.
\end{proof}

\section{General properties of exceptional sets}\label{Aset}

In this section we will derive some general properties of exceptional sets. We first show that being exceptional is preserved by conjugation.

\begin{lema}\label{lempropriE1}
If $f\colon X \to X$ and $g\colon Y \to Y$ are topologically conjugate by a homeomorphism $\pi \colon X \to Y$ with $g\circ \pi = \pi \circ f$, then for every $A \subset Y$ we have  
$$
\pi (E^+_{f|X}(\pi^{-1}(A))) = E^+_{g|Y}(A). 
$$
\end{lema}

\begin{proof}
Given $y\in \pi(E^+_{f|X}(\pi^{-1}(A)))$, 
suppose that $y\notin E^+_{g|Y}(A)$. Then there is a subsequence $(n_k)_k$ and $y_0\in A$ such that $g^{n_k}(y)$ converge to $y_0$. By conjugation, $f^{n_k}(\pi^{-1}(y))$ converges to $\pi^{-1}(y_0)\in\pi^{-1}(A)$, which is a contradiction.
Thus, $\pi (E^+_{f|X}(\pi^{-1}(A))) \subset E^+_{g|Y}(A)$. 

The other inclusion is analogous, by conjugation.
\end{proof}

We require the following simple fact which we state without proof. 

\begin{lema}\label{lem:invsubset}
Let $W\subset X$ be a compact set  such that $f(W)=W$. 

If $A\subset X$ then
$\displaystyle
	E^+_{f|W}(A\cap W) \subset E^+_{f|X}(A)
$.
\end{lema}

In order to see how exceptional sets behave with respect to iterates, for given $A\subset X$ and $N \in\N$ let us  denote  

\begin{equation}\label{defAm}
A_N := \bigcup_{j=0}^{N-1}f^{-j}(A).
\end{equation}

\begin{lema} \label{lem:invsubsetneeew}
Let  $W\subset X$ be  a compact set  such that $f(W)=W$. 

If $A\subset W$ then $E^+_{f|W}(A) = E^+_{f^N|W}(A_N\cap W)$.
\end{lema}

\begin{proof}
Let $x \in E^+_{f|W}(A)$.  Suppose that there is $y\in \ovl{\mathcal{O}_{f^N}(x)} \cap A_N\cap W$. 
Then, there are $j_0 \in \{0,\ldots, N-1\}$ such that $y\in f^{-j_0}(A)$  and a sequence $(n_k)_{k=0}^\infty$ such that $\lim_{k\to \infty}f^{Nn_k}(x) = y$.  By continuity of $f$, we have that $\lim_{k\to \infty} f^{Nn_k+j_0}(x) = f^{j_0}(y) \in A$ 
and hence $\ovl{\mathcal{O}_{f}(x)} \cap A \neq \emptyset $, which is a contradiction. This proves that $E^+_{f|W}(A) \subset E^+_{f^N|W}(A_N\cap W)$.

Consider now $x\in E^+_{f^N|W}(A_N\cap W)$. Suppose that there exists  
$y\in \ovl{O_f(x)} \cap A$. Thus, there is a subsequence $(n_k)_{k=0}^\infty$  such that $\lim_{k\to\infty} f^{n_k}(x) = y \in A$.  We can write $n_k = N s_k + r_k$ where $0\leq r_k\leq N-1$. 
Then exist $r \in \{0, \ldots, N-1\}$ such that 
$(f^{Ns_k +r}(x))_{k=0}^\infty$ is a  subsequence such that $\lim_{k\to\infty} f^{Ns_k + r}(x) = y\in A$.  By compactness of  $W$ and because $f^{Ns_k}(x) \in W$ for all $k$, there exists a  convergent subsequence $\lim_{k\to\infty}f^{N\ell_k}(x) = v \in W$. 
By continuity of $f$ we have that 
$$
	f^r(v) = f^r\big(\lim_{k\to\infty}f^{N\ell_k}(x)\big) = \lim_{k\to\infty}f^{N\ell_k+r}(x) = y.
$$    
Thus, $\lim_{k\to\infty}f^{N\ell_k}(x) = v \in f^{-r}(y)\subset 
f^{-r}(A)$, wich is a contradiction. 
This proves the other inclusion. 
\end{proof}

For the remaining  results in this section, let $W\subset X$ be a compact set such that $f(W)=W$, let $N\in \N$ and let $R\subset W$ be a compact set satisfying
\[
	f^N(R)=R
	\quad\text{ and }\quad 
	W=\bigcup_{i=0}^{N-1}f^i(R),
\]
and let $A\subset W$ and let $A_N$ be defined as in~\eqref{defAm}.

\begin{lema}\label{lempropriE1.2} 
For all $i\in\{0, \ldots, N-1\}$, we have
\[
	f^i\big(E^+_{f^N|R}(A_N\cap R)\big) \subset 
	E^+_{f^N|f^i(R)}\big(A_N\cap f^i(R)\big).
\]	
\end{lema}
\begin{proof}
Let $y\in f^i(E^+_{f^N|R}(A_N\cap R))$. 
Then there is $x\in E^+_{f^N|R}(A_N\cap R)$ such that $f^i(x)=y$. 
Suppose, by contradiction, that there is $z\in \ovl{\mathcal{O}_{f^N|f^i(R)}(y)}\cap A_N$. 
Then, there are $j\in\{0, \ldots, N-1\}$ and a sequence $(n_k)_{k=1}^\infty$ such that $\lim_{k\to \infty} f^{Nn_k}(y) = z \in f^{-j}(A)$. 
By compactness and $f^N$-invariance of $R$, we have that there are $\tilde{x}\in R$ and a subsequence $(n_\ell)_{\ell=0}^{\infty}$ of $(n_k)_{k=0}^{\infty}$ such that 
$\lim_{\ell\to \infty}f^{Nn_\ell}(x) = \tilde{x}$. 
Note that, by continuity of $f^i$, follows that 
$$
z = \lim_{\ell\to\infty}f^{Nn_\ell}(y) = f^i(\lim_{\ell\to\infty}f^{Nn_\ell}(x)) = f^i(\tilde{x}). 
$$
In this case, $\tilde{x} \in f^{-i}(z)$ and, since $z\in f^{-j}(A)$, we have that $\tilde{x}\in f^{-(i+j)}(A)$. 
If $i+j \in\{0, \ldots, N-1\}$, then $\tilde{x}\in A_N\cap R$ which is a contradiction. 
If $i+j\geq N$, then there are $s, \iota\in \N$ such that 
$\iota \in \{0, \ldots, N-1\}$ and $i+j = sN+\iota$. 
Thus, by continuity of $f^{sN}$ and $f^N$-invariance of $R$, follows that 
$$
\lim_{\ell\to\infty}f^{N(s+n_\ell)}(x) = f^{Ns}(\tilde{x}) \in f^{-\iota}(A)\cap R.  
$$ 
It is again a contradiction. 
\end{proof}

\begin{lema} \label{lempropriE2}
$\displaystyle
	E^+_{f^N|W}(A_N\cap W) 
	= \disp\bigcup_{i=0}^{N-1}E^+_{f^N|{f^i(R)}}(A_N\cap f^i(R)).
$
\end{lema}

\begin{proof}
Observe that $f^N(f^i(R)) = f^i(f^N(R)) = f^i(R)$. If $x\in E^+_{f^N|W}(A_N\cap W)$, then $x\in f^i(R)$ for some $i\in\{0,\ldots,N-1\}$ and by $f^N$-invariance of $f^i(R)$ we have
$$
\ovl{\mathcal{O}_{f^N|{f^i(R)}}(x)} \cap (A_N \cap f^i(R)) = 
\ovl{\mathcal{O}_{f^N|{f^i(R)}}(x)} \cap A_N = 
\ovl{\mathcal{O}_{f^N}(x)} \cap A_N = 
\emptyset.
$$
Hence, 
$$
E^+_{f^N|W}(A_N\cap W) \subset \disp\bigcup_{i=0}^{N-1}E^+_{f^N|{f^i(R)}}(A_N\cap f^i(R)).
$$ 

On the other hand, let $x$ be a point in the set on the right hand side, that is, let $x \in E^+_{f^N|{f^i(R)}}(A_N\cap f^i(R))$ for some $i\in\{0,\ldots,N-1\}$ and, in particular, $x \in f^i(R)$. Again, by $f^N$-invariance of $f^i(R)$, we have  
$$
\ovl{\mathcal{O}_{f^N}(x)}\cap A_N = 
\ovl{\mathcal{O}_{f^N|{f^i(R)}}(x)}\cap A_N =
\ovl{\mathcal{O}_{f^N|{f^i(R)}}(x)}\cap (A_N \cap f^i(R)) = 
\emptyset.
$$
This finishes the proof.
\end{proof}

Finally, in this section we give a relation for the entropy of the sets $A_N$, which we will need right after.

\begin{lema}\label{lem:propriE4}
If  $h( f|_W,A) < h(f|_W)$ then $h(f^N|_R,A_N\cap R) < h(f^N|_R)$.
\end{lema}

\begin{proof}
Starting from our hypothesis,
\begin{eqnarray*}
h(f|_W) & > & h(f|_W,A)\\
  \text{ by (E2) and~\eqref{eq:W} }\quad
	& = & h( f|_W,A_N\cap R) 
	= h\Big( f|_W, A_N\cap\bigcup_{i=0}^{N-1}f^i(R)\Big)  \\
  \text{ by (E5) }\quad
	& \ge &  h( f|_W,A_N\cap R)\\
  \text{ by (E4) and~\eqref{eq:W} }\quad
             & = & \disp\frac{1}{N} h( f^N|_W,A_N \cap R)
             = \frac{1}{N}h(f^N|_R,A_N\cap R)
\end{eqnarray*}
Hence, applying Lemma~\ref{lem:simple} we obtain the claimed property.
\end{proof}

\section{Proof of the main results}\label{prova}

We first establish a preparatory result for the entropy of a continuous transformation  that can be decomposed into finite systems each being conjugate to a subshift of finite type. 

\begin{prop}\label{proph}
Let $(W,d)$ be a compact metric space and $f\colon W\to W$ a continuous transformation. 
	Let $R\subset W$ be a compact set satisfying $f^N(R)=R$ and $W=\bigcup_{i=0}^{N-1}f^i(R)$  for some $N\ge1$ and suppose that $f^N\colon R\to R$ is conjugate to a subshift of finite type. 
	
	Then for every compact set $A\subset W$ satisfying $h(f|_W,A)<h(f|_W)$ we have
	\[
		h\big(f|_W,E^+_{f|W}(A)\big) 
		= h(f|_W).
	\] 
\end{prop}

\begin{proof} 
By hypothesis, there is a subshift of finite type $\sigma\colon\Sigma_M^+\to\Sigma_M^+$ and a homeomorphism $\pi\colon \Sigma_M^+\to R$ satisfying $\pi\circ f^N=\sigma\circ\pi$. 

By hypothesis and Lemma~\ref{lem:propriE4} we have 
$h( f^N|_R,A_N\cap R) < h(f^N|_R).$ 
By the conjugation property (E1) of entropy we have 
$
h( \sigma,\pi^{-1}(A_N\cap R)) < h(\sigma).
$ 
By Theorem \ref{Dolgoteoshift}, we have that 
\[
	h\big( \sigma,E^+_{\sigma|\Sigma^+_M}(\pi^{-1}(A_N\cap R))\big) 
	= h(\sigma).
\]
From Lemma \ref{lempropriE1}  and property (E1),  we conclude   
\[
	h\big( f^N|_R,E^+_{f^N|R}(A_N\cap R)\big) 
	= h(f^N|_R).
\]
By $f^N$-invariance of $R$, properties  (E2) and (E5) of  entropy, Lemma \ref{lempropriE1.2}, Lemma \ref{lempropriE2}, and Lemma \ref{lem:simple}
\[\begin{split}
h(f^N|_R) & =  h\big( f^N|_W,E^+_{f^N|R}(A_N\cap R)\big) \\
        & =  h\big( f^N|_W,f^i(E^+_{f^N|R}(A_N\cap R))\big) \\
        & \leq  h\big( f^N|_W,E^+_{f^N|f^i(R)}(A_N\cap f^i(R))\big)\\
        & \leq  h(f^N|_R).
\end{split}\]
Thus, by Lemma \ref{lempropriE2}, $f^N$-invariance of $R$,  (E3), and Lemma \ref{lem:simple}, it follows  
\[\begin{split}
h\big( f^N|_W,E^+_{f^N|W}(A_N\cap W)\big) 
	& = \max_{0\leq i \leq N-1} h\big( f^N|_{f^i(R)},E^+_{f^N|{f^i(R)}}(A_N \cap f^i(R))\big)\\
	& = h(f^N|_{R})\\ 
	& =  N h(f|_W). 
\end{split}\]
Then, by Lemma~\ref{lem:invsubsetneeew} and property (E4), it follows that
\[\begin{split}
h(f|_W,E^+_{f|W}(A))
  & =  h\big( f|_W,E^+_{f^N|W}(A_N\cap W)\big) \\
 & =  \frac{1}{N} h\big( f^N|_W,E^+_{f^N|W}(A_N\cap W)\big) \\ 
 & =  h(f|_W), 
\end{split}\]
and this finishes the proof of the proposition. 
\end{proof}

Now we can give the proofs of Theorem \ref{teoentropy} and Theorem  \ref{teoprinc}. 

\begin{proof}[Proof of Theorem \ref{teoentropy}]
By hypothesis, $h(f|_J)>0$. 
By the variational principle and Ruelle's inequality, for every $\epsilon>0$ there is an ergodic measure $\mu$ satisfying $\chi(\mu)>0$ and 
$h_\mu(f) \ge h(f|_J)  - \epsilon$. 
By Theorem \ref{Katrin1}, there is a compact set $W_\epsilon\subset J$ such that 
$h(f|_{W_\epsilon})\ge h_\mu(f) - \epsilon$. 
If $\epsilon$ was sufficiently small, this and our hypothesis  $h(f|_J, A) < h(f|_J)$ together imply  $h(f|_{W_\epsilon}, A\cap W_\epsilon) < h(f|_{W_\epsilon})$.
Then, by Proposition~\ref{proph}, the above inequalities, and observing that $E^+_{f|{W_\epsilon}}(A\cap W_\epsilon)\subset E^+_{f|J}(A)$, we have   
\[\begin{split}
	h(f|_J) & \leq  h_{\mu}(f) + \epsilon \le h(f|_{W_\epsilon})+2\epsilon\\
       & =  h( f|_{W_\epsilon},E^+_{f|W_\epsilon}(A\cap W_\epsilon)) + 2\epsilon\\
       & \leq  h( f|_J,E^+_{f|J}(A)) + 2\epsilon.
\end{split}\]  
Since $\epsilon$ was arbitrary,  this implies the claim.
\end{proof}

\begin{proof}[Proof of Theorem \ref{teoprinc}]
Consider the sequences $(\mu_n)_n$, $(\epsilon_n)_n$ and $(W_{n})_n$ from Lemma \ref{lemaqeaprox}. In particular, $\epsilon_n<\chi(\mu_n)/n$ and thus
\[
	\lim_{n\to\infty}\frac{\chi(\mu_n)-\epsilon_n}
		{\chi(\mu_n)+\epsilon_n}
	=1.
\]
By  hypothesis we have 
\[
	\dim_{\rm H}A<\DD(f|_J).
\]
Hence, for $n$ sufficiently large we have (the first inequality is simple)
\[
	\dim_{\rm H}(A\cap W_{n})
	\le \dim_{\rm H}A
	<\frac{\chi(\mu_n)-\epsilon_n}{\chi(\mu_n)+\epsilon_n}\dim_{\rm H}W_{n}
	\leq \DD	(f|_J).
\]
Applying Corollary~\ref{proplema 2.0cor}, the above inequality, and again Corollary~\ref{proplema 2.0cor}, we obtain
\[\begin{split}
	h( f|_{W_{n}},A\cap W_{n}) 
	& \le  (\chi(\mu_n) + \epsilon_n)\dim_{\rm H} (A\cap W_{n})\\ 
 & <  (\chi(\mu_n) - \epsilon_n)\dim_{\rm H}  W_{n}\\
  &\leq  h(f|_{W_{n}}).  
\end{split}\]
Hence, we can apply Proposition~\ref{proph} and obtain 
\[
	h\big( f|_{W_{n}},E^+_{f|W_{n}}(A\cap W_{n})\big) 
	= h(f|_{W_{n}}). 
\]
Together with  Lemma \ref{afirm2.1.1} applied to $W=W_{n}$ and $Y= E^+_{f|W_{n}}(A\cap W_{n})$ this implies
\[\begin{split}
	\dim_{\rm H} E^+_{f|W_{n}}&(A\cap W_{n})\\
	&\ge 
	\frac{h\big( f|_{W_{n}},
		E^+_{f|W_n}(A\cap W_{n})\big)}
	{h(f|_{W_{n}})}\frac{(\chi(\mu_n) -\epsilon_n)}
		{(\chi(\mu_n) +\epsilon_n)}\dim_{\rm H}W_{n}\\
	&=\frac{(\chi(\mu_n) -\epsilon_n)}
		{(\chi(\mu_n) +\epsilon_n)}\dim_{\rm H}W_{n}	.
\end{split}\]
Lemma \ref{lemaqeaprox} now proves that  
\[
\lim_{n\to\infty} \dim_{\rm H} E^+_{f|W_{n}}(A\cap W_{n}) 
\geq \DD(f|_J)). 
\]
Observe that, by Lemma~\ref{lem:invsubset}
it follows that 
\[
	E^+_{f|W_{n}}(A\cap W_{n}) \subset E^+_{f|J}(A).
\]	 
Now property (H1) of the Hausdorff dimension implies  
$$
\dim_{\rm H} E^+_{f|J}(A) \geq \DD(f|_J)
$$
and proves the theorem.
\end{proof}

\bibliographystyle{amsplain}

\end{document}